\documentclass[10pt]{article}
\usepackage{amssymb}
\usepackage[pdftex]{graphicx}
\usepackage{xcolor} 
\usepackage{tensor}
\usepackage{fullpage} 
\usepackage{amsmath}
\usepackage{amsthm}
\usepackage{verbatim}

\usepackage{enumitem}
\setlist[enumerate]{leftmargin=1.5em}
\setlist[itemize]{leftmargin=1.5em}

\usepackage{seqsplit,mathtools} 
\usepackage{caption}
\usepackage{subcaption}

\definecolor{green}{rgb}{0,0.8,0} 

\newcommand{\Red}[1]{\begingroup\color{red} #1\endgroup} 
\newcommand{\Blue}[1]{\begingroup\color{blue} #1\endgroup} 

\newtheorem{thm}{Theorem}[section]
\newtheorem{cor}[thm]{Corollary}
\newtheorem{lem}[thm]{Lemma}

\newtheorem{defn}[thm]{Definition}

\theoremstyle{definition}

\theoremstyle{remark}
\newtheorem{rmk}[thm]{Remark}

\numberwithin{equation}{section}

\newcommand{\tld}[1]{\widetilde{#1}}
\newcommand{\br}[1]{\overline{#1}}

\newcommand{\nnrm}[1]{{\vert\kern-0.25ex\vert\kern-0.25ex\vert #1 
		\vert\kern-0.25ex\vert\kern-0.25ex\vert}}

\newcommand{\supp}{{\mathrm{supp}}\,}

\newcommand{\lap}{\Delta}

\newcommand{\rd}{\partial}
\newcommand{\nb}{\nabla}

\newcommand{\ift}{\infty}

\newcommand{\alp}{\alpha}
\newcommand{\bt}{\beta}
\newcommand{\gmm}{\gamma}
\newcommand{\Gmm}{\Gamma}
\newcommand{\dlt}{\delta}

\newcommand{\Lmb}{\Lambda}

\newcommand{\omg}{\omega}
\newcommand{\Omg}{\Omega}

\newcommand{\zt}{\zeta}

\newcommand{\bbN}{\mathbb N}

\newcommand{\bbR}{\mathbb R}

\newcommand{\bbT}{\mathbb T}

\newcommand{\bbZ}{\mathbb Z}

\newcommand{\calT}{\mathcal T}

\newcommand{\To}{\longrightarrow}

\begin{document}

\bibliographystyle{plain}
\title{\Large Stability of monotone, non-negative, and compactly supported vorticities \\ in the half cylinder and infinite perimeter growth for patches}
\author{Kyudong Choi\thanks{Department of Mathematical Sciences, Ulsan National Institute of Science and Technology, 50 UNIST-gil, Eonyang-eup, Ulju-gun, Ulsan 44919, Republic of Korea. Email: kchoi@unist.ac.kr}\and In-Jee Jeong\thanks{Department of Mathematical Sciences and RIM, Seoul National University, 1 Gwanak-ro, Gwanak-gu, Seoul 08826, Republic of Korea. Email: injee$ \_ $j@snu.ac.kr}\and 
	Deokwoo Lim\thanks{Department of Mathematical Sciences, Ulsan National Institute of Science and Technology, 50 UNIST-gil, Eonyang-eup, Ulju-gun, Ulsan 44919, Republic of Korea. Email: dwlim@unist.ac.kr} }

\date\today
\maketitle
\renewcommand{\thefootnote}{\fnsymbol{footnote}}
\renewcommand{\thefootnote}{\arabic{footnote}}

\begin{abstract}
	We consider the incompressible Euler equations in the half cylinder $ \bbR_{>0}\times\bbT $. In this domain, any vorticity which is independent of $x_2$ defines a stationary solution. We prove that such a stationary solution is nonlinearly stable in a weighted $L^{1}$ norm involving the horizontal impulse, if the vorticity is non-negative and non-increasing in $x_1$. This includes {stability of cylindrical patches} 
	$ \lbrace x_{1}<\alp\rbrace,\; \alp>0 $. The stability result is based on the fact that such a profile is the unique minimizer of the horizontal impulse among all functions with the same distribution function. 
	Based on stability, we prove existence of vortex patches in the half cylinder that exhibit infinite perimeter growth in infinite time.
\end{abstract}\vspace{1cm}

\section{Introduction}

We consider the incompressible Euler equations in the half cylinder $ S_{+}:= \bbR_{> 0}\times \bbT $, where $ \bbT:=[-\pi,\pi)
$ is the torus, in vorticity form:
\begin{equation}\label{Eulereqs}
	\begin{split}
		\rd_{t}\omg+u\cdot\nb\omg&=0\quad\text{for}\quad (t,x)\in (0,\ift)\times S_{+},\\
		\omg|_{t=0}&=\omg_{0}\quad\text{for}\quad x\in S_{+}.
	\end{split}
\end{equation}
The velocity $ u $ in \eqref{Eulereqs} is determined from the vorticity $ \omg $ by the cylindrical Biot--Savart law, {imposing no-flow condition at the boundary $ \rd S_{+}:=\lbrace0\rbrace\times\bbT $, where $ u $ vanishes as $ x_{1} $ goes to infinity. The exact form of the Biot--Savart law will be discussed in Section 2.1.}
It can be shown that for an initial data $ \omg_{0}\in  L^{\ift}(S_{+}) $ with bounded 
support, which means it is compactly supported so $\omega_0  \in (L^{1}\cap L^{\ift}) $, there exists a unique global-in-time weak solution $ \omg\in L^{\ift}\big(0,\ift;(L^{1}\cap L^{\ift})(S_{+})\big) $ of \eqref{Eulereqs}, by following the arguments of Yudovich \cite{Yu63} in $ \bbR^{2} $.  By using the result of Kelliher \cite{Kelliher15}, Beichman--Denisov in the appendix of \cite{BeD17} demonstrated this in the full cylinder $ S:=\bbR\times\bbT $, which is a general case of $ S_{+} $ having an odd-symmetry in $ x_{1} $ .

\medskip

In this paper, we present two results: 
the stability of a compactly supported non-negative, monotone, and $ x_{2}- $independent vorticity, and the existence of patch-type solutions which exhibit infinite growth of perimeter in infinite time. This seems to be the first infinite perimeter growth result for patches defined in an unbounded domain.


\subsection{Main results}

We denote the weighted $ L^{1}- $norm on $ S_{+} $ with weight $ 1+x_{1} $ as the $ J_{1}- $norm on $ S_{+} $:
\begin{equation*}
	\big\|f\big\|_{J_{1}(S_{+})}:
	=\int_{S_{+}}(1+x_{1})\big|f(x)\big|dx.
\end{equation*}

Note that any $ x_{2}- $independent function is a stationary solution of \eqref{Eulereqs} in $ S_{+} $, since the $ x_{2}- $independence $ \rd_{x_{2}}\omg\equiv0 $ of $ \omg $ leads to $ u_{1}\equiv0 $, which gives us
\begin{equation*}
	u\cdot\nb\omg=
	u_{1}\rd_{x_{1}}\omg+u_{2}\rd_{x_{2}}\omg=
	0\cdot\rd_{x_{1}}\omg+u_{2}\cdot0=0.
\end{equation*}
The first result of this paper is about the stability of such a $ x_{2}- $independent solution when it  is non-negative and monotone with bounded support. Theorem \ref{thm_stabi} is analogous to that of radial and monotone solution in $ \bbR^{2} $ from the recent work \cite{ChL21}.

\begin{thm}\label{thm_stabi}
	For any constants $ L,M>0 $, there exists a constant $ C_{1}=C_{1}(L,M)>0 $ such that if $ \zt=\zt(x_{1}) $ is in $ L^{\ift}(S_{+}) $, non-negative, non-increasing, and compactly supported with
	\begin{equation}\label{eq_LM}
		\supp(\zt)\subset \lbrace x_{1}<L\rbrace
		,\quad \big\|\zt\big\|_{L^{\ift}(S_{+})}\leq M,
	\end{equation}
	then for any non-negative $ \omg_{0}\in L^{\ift}(S_{+}) $ with compact support
	, the corresponding solution $ \omg(t) $ of $ \eqref{Eulereqs} $ satisfies
	\begin{equation}\label{eq_J1estimate}
		\sup_{t\geq0}\big\|\omg(t)-\zt\big\|_{J_{1}(S_{+})}\leq C_{1}\big[\big\|\omg_{0}-\zt\big\|_{J_{1}(S_{+})}^{\frac{1}{2}}+\big\|\omg_{0}-\zt\big\|_{J_{1}(S_{+})}\big].
	\end{equation}
\end{thm}
\begin{rmk}\label{rmk_necessity}
	The non-negativity of $ \omg_{0} $ is necessary in our proof for technical reasons. One of the reasons is that the rearrangement of a function, which is frequently used in our proof, cannot be defined on the infinite domain $ S_{+} $ if the function contains a negative part. On the other hand, boundedness and compact support of $ \omg_{0} $ 
	are imposed to apply the standard  global well-posedness theory as well as to ensure that $J_1$-norm of $\omega_0$ is finite. 
\end{rmk}
This can be considered as  an extension of the work by Marchioro--Pulvirenti \cite{MaP85}, where they considered the Euler equations \eqref{Eulereqs}
in a \textit{bounded} strip  while we consider the equations in the \textit{infinite} strip. 
More precisely, the result in \cite[Theorem 1]{MaP85} shows the explicit $ L^{1}- $stability of the form \eqref{eq_J1estimate} of the $ x_{2}- $independent and monotone $ \zt\in L^{\ift} $ in a bounded strip $ \lbrace x_{1}< L\rbrace $, where $ C_{1} $ depends on the domain size $ L $ and the $ L^{\ift}- $norm of the initial perturbed data $ \omg_{0} $. By imposing the weight $ 1+x_{1} $ on the $ L^{1}- $norm, we were able to remove the dependence on $ \big\|\omg_{0}\big\|_{L^{\ift}} $  and to extend the domain into the infinite strip $ \lbrace x_{1}<\ift\rbrace $.

\medskip

\noindent This  stability result is used in showing our next result, which deals with the perimeter of a smooth vortex patch. In $ \bbR^{2} $, if the initial data is given as a patch $ \omg_{0}=1_{\Omg_{0}} $ for some open set $ \Omg_{0}\subset\bbR^{2} $, then the unique solution is of the form $ \omg(t)=1_{\Omg_{t}} $, where $ \Omg_{t}=\Phi(t,\Omg_{0}) $. Here, $ \Phi(t) $ is the flow map, which is the unique solution of the ODE
\begin{equation}\label{eq_flowmapext}
	\begin{split}
		\frac{d}{dt}\Phi(t,x)&=u\big(t,\Phi(t,x)\big), \qquad \Phi(0,x) =x,
	\end{split}
\end{equation}
where $ u(t ) $ is determined from $ \omg(t ) $ by the Biot--Savart law in $ \bbR^{2} $, given as $ u(t )=-\frac{1}{2\pi}\frac{x^{\perp}}{|x|^{2}}\ast\omg(t ) $. Moreover, it is well known if the boundary $ \rd\Omg_{0} $ of $ \Omg_{0} $ is connected and $C^\infty$--smooth, then $ \rd\Omg_{t} $ is connected and $C^\infty$--smooth as well (\cite{Chemin,BeCo93,Serfati}). It can be shown that the same holds for open patches with smooth boundaries in $ S_{+} $, by adapting the proof of Kiselev--Ryzhik--Yao--Zlat\v{o}s \cite{KRYZ16}.

\medskip

\noindent 
Recently, it was shown in \cite{CJ_perimeter} that there exists a patch on $ \bbR^{2} $ that has perimeter growth for finite time. The following theorem shows the existence of a patch on $ S_{+} $ where the growth of the perimeter is \textit{infinite} in infinite time (see Figure \ref{fig:patch}).


\begin{thm}\label{thm_perim}
	There exists a constant $ C_{2}>0 $ and a bounded, open 
	set $ \Omg_{0}\subset S_{+} $ with its smooth, connected 
	boundary $ \rd\Omg_{0}\subset \br{S_{+}}:=\lbrace x_{1}\geq0\rbrace $ that satisfies
	\begin{equation*}\label{eq_Omgbdry}
		\text{length}(\rd\Omg_{0})\leq 20,
		\quad \rd\Omg_{0}\cap\rd S_{+}\neq\emptyset,
	\end{equation*}
	such that for any $ t\geq0 $, 
	the solution $ 1_{\Omg_{t}} $ of \eqref{Eulereqs} with the initial data $ 1_{\Omg_{0}} $ satisfies
	\begin{equation*}\label{eq_iftgrowthofper}
		\text{length}(\rd\Omg_{t})\geq C_{2}t,\quad \forall t\geq0.
	\end{equation*}
	Here, the boundary $\partial S_+$ of $ S_{+} $ is the circle $ \lbrace 0\rbrace\times\bbT $. 
\end{thm}
In an annulus, suggested by Nadirashvili \cite{Nadirashivili91}, two points each on the inner circle and the outer circle possess different angular velocities for all time. This difference leads to the growth of length of the curve on the annulus that connects the two points. Theorem \ref{thm_perim} shows that such growth of length can occur in an unbounded space having one connected boundary component as well.\\

\begin{figure}
	\centering
	\includegraphics[scale=0.25]{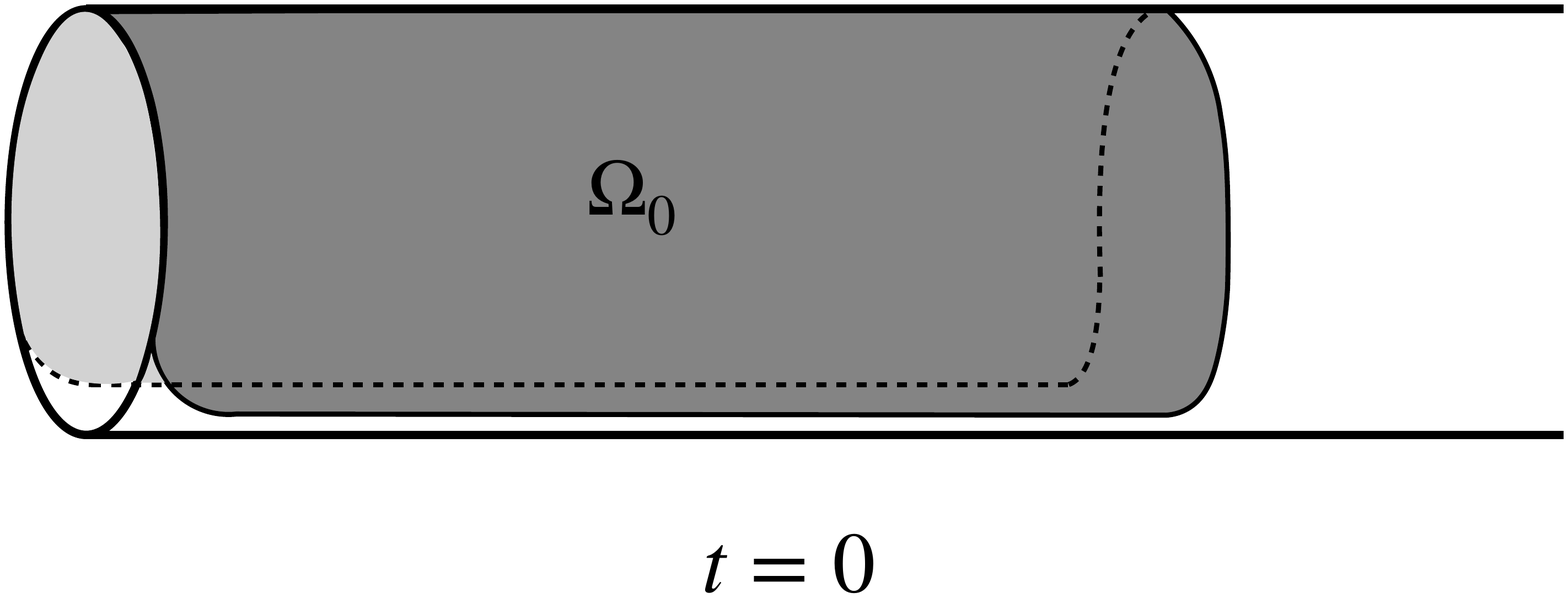} \, \, \, \includegraphics[scale=0.25]{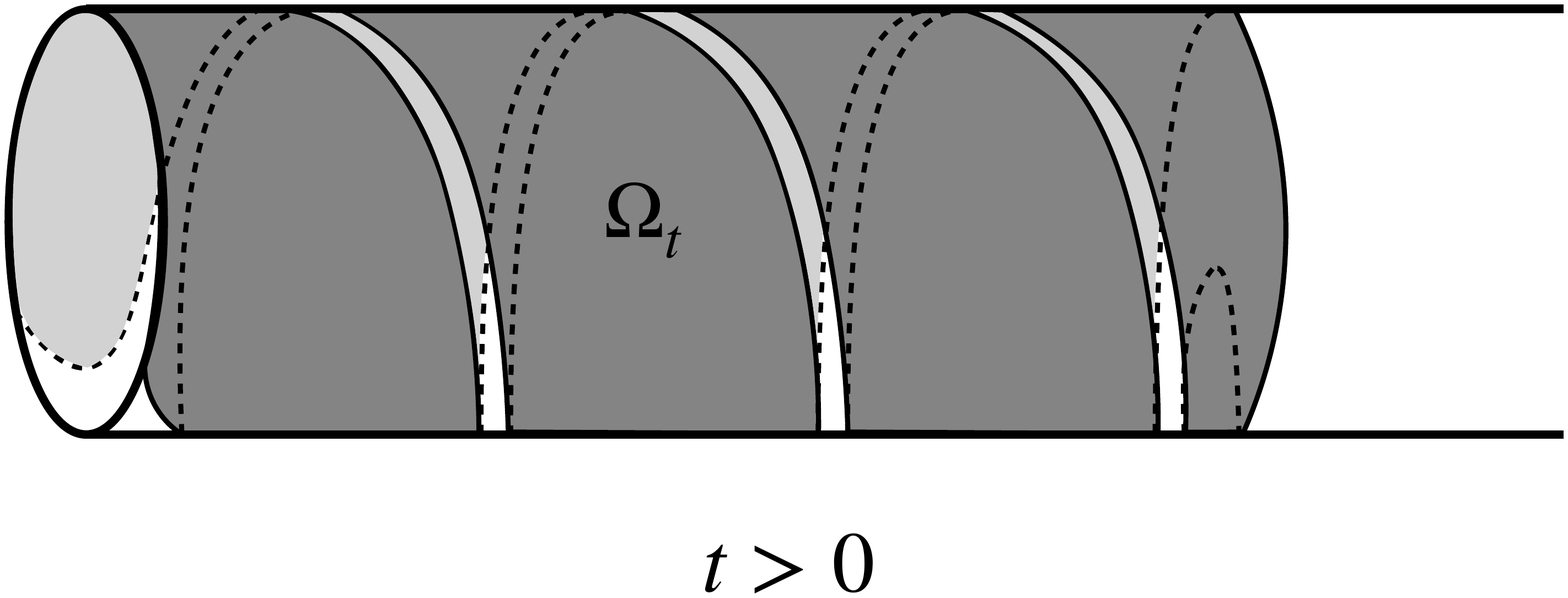}
	\caption{A schematic diagram for perimeter growth in Theorem \ref{thm_perim}} \label{fig:patch}
\end{figure}

Wan--Pulvirenti \cite{WaP85} showed a $ L^{1}- $stability of a disc patch $ 1_{B_{r}},\ r>0 $ among patches in a bounded disc $ B_{R},\ R\geq r $. More generally, \cite{MaP85} proved a $ L^{1}- $ stability of a radially symmetric and monotone solution among solutions in $ L^{\ift}(B_{R}),\ R>0 $ as well. Sideris--Vega \cite{SiV09} presented a $ L^{1}- $stability of a disc patch $ 1_{B_{r}},\ r>0 $ among patches with a compact support in $ \bbR^{2} $.  The most recent work \cite{ChL21} extended the work in \cite{MaP85} to show a weighted $ L^{1}- $stability of a non-negative, radially symmetric, and non-increasing solution among non-negative solutions in $ L^{\ift}(\bbR^{2}) $. In addition, \cite{AbC19}   proved an orbital stability, admitting translation, of the Lamb dipole by using a variational method. \cite{Ch20} showed a stability, up to translation, of the Hill's spherical vortex among non-negative  axi-symmetric solutions in 3D. \cite{ChJe22_3} presented an orbital stability in $ L^{1} $ of $ m- $fold Kelvin waves for symmetric perturbations and constructed an $ m- $fold Kelvin wave which shows perimeter growth for sufficiently long finite time.

\medskip

\noindent 
In the full strip $S$, \cite{BeD17} 
showed a stability of a cylindrical  patch $ 1_{\lbrace |x_{1}|<L\rbrace} $ for sufficiently large $ L>2 $. The existence and uniqueness  result in the strip is based on the work from 
\cite{Kelliher15}, in which it characterized various properties of a solution with bounded velocity and vorticity. \cite{ChDe19} estimated the upper bound of the horizontal support size of a non-negative solution in $ L^{\ift}(S) $ with compact support in $ S $. Bedrossian--Masmoudi \cite{BeMa15} proved asymptotic stability and inviscid damping of  Couette  flow in $ S $. 
Zillinger \cite{Zillinger17} showed linear inviscid damping for monotone shear flows.
Ionescu--Jia \cite{IonJia_cmp} showed inviscid damping of a shear flow in a bounded cylinder $ [0,1]\times\bbT $, along with that the support of the vorticity perturbation stays in the cylinder for all $ t\geq0 $. 

\medskip

\noindent \cite{Nadirashivili91} proved that in an annulus, there exists a smooth solution of \eqref{Eulereqs} which is $ C^{1}- $unstable among smooth solutions. \cite{CJ_perimeter} constructed a patch in $ \bbR^{2} $ in which the perimeter of the patch boundary grows for sufficiently large finite time. 
 \cite{ChJe22} constructed a solution, close to Lamb dipole, with compact support in $ \bbR^{2} $ which shows growth of its gradient and support size for infinite time. \cite{ChJe22_2} proved the existence of a solution that is close to the Hill's vortex in $ L^{1} (\mathbb{R}^3)$, its support being close to that of the Hill's vortex as well, which exhibits filamentation for infinite time. This, as corollaries, produced infinite growth of perimeter and gradient as well.

\subsection{Key ideas}

In order to prove Theorem \ref{thm_stabi}, we needed to reproduce the main ideas that were used in the proof of \cite[Lemma 3.3]{ChL21}, such as the conservation of the measure of level sets, the total mass, the angular impulse, and properties of symmetric decreasing rearrangement, including the rearrangement estimate established in 
\cite{MaP85}. We use the conservation of the level set measure, the total mass, and the horizontal impulse. The last quantity corresponds to the angular impulse in $ \bbR^{2} $.
Then we define the 
rearrangement on $ S_{+} $ to be $ x_{2}- $independent so that it is defined in the same way the symmetric decreasing rearrangement is defined on $ \bbR^{2} $. This implies that the nonexpansivity property and the rearrangement estimate on $ \bbR^{2} $ can be reproduced to $ S_{+} $. 
{In addition, we use the idea of \textit{cutting off} a non-negative function by a certain height. One of the important property of the cut-off operator $ \Gmm_{\alp},\ \alp>0 $, which is defined in section 2.4, is that it commutes with the rearrangement operator $ \ast $: $$ (\Gmm_{\alp}f)^{\ast}=\Gmm_{\alp}(f^{\ast}), $$ for any non-negative $ f\in L^{1} $. The use of this cut-off operator makes our stability result be independent of the $ L^{\ift}- $norm of $ \omg_{0} $.}
\begin{rmk}
	Using the conservation of the total mass and the horizontal impulse, we can reproduce the simplest stability of the stationary patch $ 1_{D} $ of the unit disc $ D{=B_{1}(0)} $ in $ \bbR^{2} $ from \cite{SiV09} and Dritschel \cite{Dr88} to the stationary patch $ 1_{\br{\Omg}} $ with $$ \br{\Omg}:=\lbrace x_{1}<1\rbrace\subset S_{+} :$$
	\begin{equation*}\label{eq_simpstabi}
		\begin{split}
			\int_{\Omg_{t}\triangle \br{\Omg}}|1-x_{1}|dx&=\int_{\Omg_{t}\setminus \br{\Omg}}(1-x_{1})dx+\int_{\br{\Omg}\setminus\Omg_{t}}(x_{1}-1)dx=\int_{\Omg_{t}}(1-x_{1})dx+\int_{\br{\Omg}}(x_{1}-1)dx\\
			&=\int_{\Omg_{0}}(1-x_{1})dx+\int_{\br{\Omg}}(x_{1}-1)dx=\int_{\Omg_{0}\triangle \br{\Omg}}|1-x_{1}|dx,\quad t\geq0.
		\end{split}
	\end{equation*} However, note that the weight $|1-x_1|$ vanishes when $x_1=1$; this norm is not strong enough to control the flow map pointwise. 
\end{rmk}

\medskip

\noindent We move on to Theorem \ref{thm_perim}. The idea in \cite{CJ_perimeter} was using the $ L^{1}- $stability of the circular patch $ 1_{D} $ 
and its corresponding velocity $ u_{D} $, which can be derived explicitly. The fact that 
the {Biot--Savart} 
kernel $ K $ in $ \bbR^{2} $ is roughly $ \frac{1}{|x|} $ was used to show that the difference between velocities $ u_{\Omg_{t}} $ of $ 1_{\Omg_{t}} $ and $ u_{D} $ of $ 1_{D} $ is uniformly bounded by the $ L^{1}- $difference of $ 1_{\Omg_{0}} $ and $ 1_{D} $ up to the power of $ \frac{1}{4} $. 
The method used in the proof required two points 
that are 
located on $ \rd\Omg_{0} $ in the initial time and move along the flow map to have certain amount of radial distance with each other. 
If that is the case, then the tangential velocities of those two points differ. This difference creates the perimeter growth of the boundary.
However, for infinite time, there is no reason for the radial distance {in the future} to be nonzero, since the radial velocity of each point need not {vanish}. Thus, in this method, the perimeter growth of $ \rd\Omg_{t} $ is limited to finite time where the distance is kept. This is why we use a different approach in Theorem \ref{thm_perim}.

\medskip 

\noindent Similarly as in $ \bbR^{2} $, the unique solution of \eqref{Eulereqs} with patch-type initial data $ \omg_{0}=1_{\Omg_{0}} $ is the patch $ \omg(t)=1_{\Omg_{t}} $, where $ \Omg_{t} $ is the image of $ \Omg_{0} $ through the flow map. 
As a way of understanding this 
patch $ 1_{\Omg_{t}} $, 
we can go through the following process. First, we {periodically} extend the velocity field $ u(t) $ of $ 1_{\Omg_{t}} $ to the half plane $ \bbR_{+}^{2}:=\bbR_{>0}\times\bbR  $. Then, denoting this extension as $ u_{ext}(t) $, we generate the flow map $ \Phi(t) $ in $ \bbR_{+}^{2} $ from $ u_{ext}(t) $. Then we project $ \Omg_{0}\subset S_{+} $ through the flow map $ \Phi(t) $ to obtain the image $ \Phi(t,\Omg_{0})=\Phi_{t}(\Omg_{0}) $. Finally, $ \Omg_{t}\subset S_{+} $ is obtained by projecting $ \Phi_{t}(\Omg_{0}) $ onto $ S_{+} $ via the map $ Q : \bbR_{+}^{2} \To S_{+} $, which is defined as $$ Q(x_{1}, x_{2}):=(x_{1}, x_{2}-2n\pi), $$ for some $ n\in\bbZ $ that satisfies $ x_{2}\in[(2n-1)\pi,(2n+1)\pi) $. Such a framework is depicted in Figure \ref{fig_patches}.

\begin{figure}
	\centering
	\includegraphics[scale=0.52]{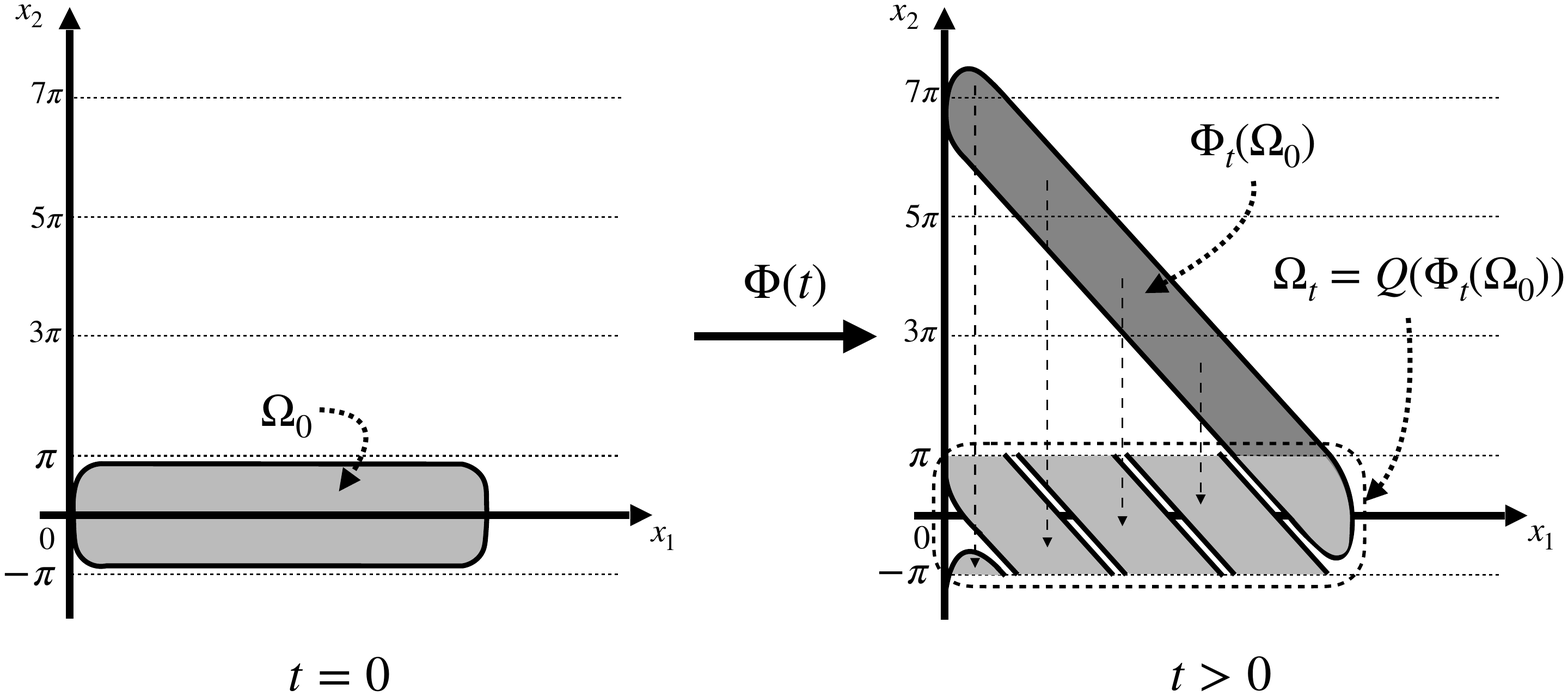}
	\vspace{-1cm}
	\caption{A schematic diagram of understanding $ \Omg_{t}\subset S_{+} $} \label{fig_patches}
\end{figure}

\medskip
We require the boundary $ \rd\Omg_{0} $ of the initial patch $ \Omg_{0} $ to be \textit{connected} and to have at least one point that \textit{intersects with} $ \rd S_{+}=\{0\}\times\mathbb{T} $, the boundary of the domain $ S_{+} $ (cf.  the boundary $\partial \br\Omega=\{0,1\}\times \mathbb{T}$ of the (steady) set $\br\Omega$ is not connected.).
Then we use the fact that the point 
stays on {the $ x_{2}- $axis} $ \rd \bbR_{+}^{2} $ along the flow map, that is, $ \rd\Phi_{t}(\Omg_{0})\cap\rd \bbR_{+}^{2}\neq\emptyset $ for any $ t\geq0 $. This is because the horizontal velocity of any point on $ \rd \bbR_{+}^{2} $ is zero. 
Then we use the important fact that the vertical velocity of that point is strictly greater than the growth rate of the vertical center of mass of the patch $ 1_{\Phi_{t}(\Omg_{0})} $ on $ \bbR_{+}^{2} $:
$$ u_{2}\big(t,\Phi(t,x_{0})\big)-\frac{d}{dt}\frac{1}{|{\Phi_{t}}(\Omg_{0})|}\int_{\Phi_{t}(\Omg_{0})}x_{2}dx\geq C>0,\quad x_{0}\in\rd\Omg_{0}\cap\rd S_{+},\quad t\geq0, $$
which is depicted in Lemma \ref{lem_comgrowth}. In this lemma, we show that the rate of $ 1_{\Phi_{t}(\Omg_{0})} $ is close to the rate of $ 1_{\br{\Phi}_{t}(\br{\Omg})} $, 
the half of the vertical velocity of the point on $ {\rd\br{\Phi}_{t}(\br{\Omg})\cap}\rd\bbR_{+}^{2} $. {Here, $ \br{\Phi}(t) $ is the flow map induced from the periodic extension of $ \br{u} $, the velocity of the steady solution $ 1_{\br{\Omg}} $.} 
Due to this difference in vertical velocities, which is maintained for \textit{every} time, the vertical distance between the point on $ \rd\Phi_{t}(\Omg_{0})\cap\rd\bbR_{+}^{2} $ and some point on $ \rd\Phi_{t}(\Omg_{0}) $ near the vertical center of mass of $ 1_{\Phi_{t}(\Omg_{0})} $
is bigger than $Ct$. This distance works as a lower bound of $ \text{length}(\rd\Omg_{t}) $ in Theorem \ref{thm_perim}.
To prove this, we 
use the stability \eqref{eq_J1estimate} from Theorem \ref{thm_stabi} applied on the bounded strip $ 1_{\br{\Omg}} $, and the velocity $ \br{u} $ corresponding to $ 1_{\br{\Omg}} $ that is calculated explicitly in Section 3.1.

\subsection*{Organization of the paper} 
In Section 2, we define the cylindrical Biot--Savart law and list the conserved quantities of $ \omg(t) $, give the definition of rearrangement and its properties, define the cut-off operator, and move on to the proof of Theorem \ref{thm_stabi}. We begin Section 3 by calculating the velocity field $ \br{u} $ of the patch $ 1_{\br{\Omg}} $ and the total mass of $ 1_{\br{\Omg}} $. 
After defining the map $ Q $ from the , we define periodic extensions of $ u(t) $ and $ \br{u} $, and induce flow maps $ \Phi(t) $ and $ \br{\Phi}(t) $ from these extensions. Then we present that $ \Omg_{t} $ is the image of $ \Phi(t,\Omg_{0})=\Phi_{t}(\Omg_{0}) $ through the map $ Q $. 
After explicitly calculating the growth rate of the vertical center of mass of the patch $ 1_{\br{\Phi}_{t}(\br{\Omg})} $ on $ \bbR_{+}^{2} $, 
we prove several lemmas concerning the control of the difference of vertical velocities and the growth rates of vertical center of mass between $ 1_{\Phi_{t}(\Omg_{0})} $ and $ 1_{\br{\Phi}_{t}(\br{\Omg})} $. 
We finish the paper with the proof of Theorem \ref{thm_perim}. 

\section{Stability}

\subsection{The cylindrical Biot--Savart law}

We shall specify the exact form of the cylindrical Biot--Savart law, {considered in this paper,} which gives the velocity $ u $ in \eqref{Eulereqs} from the vorticity $ \omg $. We recall the Biot--Savart law for the full cylinder $ S=\bbR\times\bbT $ (\cite[App. A]{BeD17}). The kernel $ K $ for $ S_{+} $ can be derived simply by imposing the odd symmetry in $ x_{1} $ to $ \omg(x_1,x_2)$. 
 Then 
$ u $ is given as (see \cite{BeD17} for detail)
\begin{equation}
	\begin{split}
		u(x)&=\nb^{\perp}\Psi(x)
		=\int_{S_{+}}K(x,y)
		\omg(y)dy,
		\label{BiotSavart}
	\end{split}
\end{equation}
where $ \Psi $ is the stream function defined as
\begin{equation*}\label{eq_streamftn}
	\begin{split}
		\Psi(x)&=\int_{S_{+}}\Gmm(x,y)\omg(y)dy,\quad\Gmm(x,y)=\Gmm_{S}(x-y)-\Gmm_{S}(x+\br{y}),\\
		\Gmm_{S}(x)&=-\frac{1}{4\pi}\ln(\cosh x_{1}-\cos x_{2}),\quad\br{y}=(y_{1},-y_{2}),
	\end{split}
\end{equation*}
and the kernel $ K=(K_{1},K_{2}) $ is given by
\begin{equation}\label{eq_kernelK}
	\begin{split}
		K(x,y)&=\nb_{x}^{\perp}\Gmm(x,y)=
		K_{S}(x-y)-K_{S}(x+\br{y}),\\
		K_{S}(x)&=\nb^{\perp}\Gmm_{S}(x)=\frac{(\sin x_{2},\,-\sinh x_{1})}{4\pi(\cosh x_{1}-\cos x_{2})}.
	\end{split}
\end{equation}
Here, when $ \omg\in L^{\ift}(S_{+}) $ with bounded support is considered, $ \Psi $ solves the elliptic problem
\begin{equation}\label{eq_elliptic}
	-
	\lap\Psi=\omg\quad\mbox{in}\quad S_+,\quad \Psi|_{x_{1}=0}\equiv0,\quad \lim\limits_{x_{1}\to\ift}\rd_{x_{1}}\Psi(x)=0,\quad |\Psi(x)|\leq C(x_{1}+1),
\end{equation}
and $ K_{S} $ is the cylindrical Biot--Savart kernel in $ S $. 
The kernel $ K $ in $ S_{+} $ has the form \eqref{eq_kernelK} due to the odd-symmetry of $ \omg $ in $ x_{1} $. Also,
the kernel $ K $ has the following rough estimate
\begin{equation}\label{eq_kernelsest}
	|K_{1}(x,y)|\leq \frac{C_{0}}{|x-y|},\quad |K_{2}(x,y)|\leq \frac{C_{0}}{|x-y|}+1,\quad \text{if}\quad -\pi\leq x_{2}-y_{2}\leq\pi,
\end{equation}
with some constant $ C_{0}>0 $.

\subsection{Conserved quantities}

For a function $ f $ defined on $ S_{+} $ where both $ f $ and $ x_{1}f $ are in $ L^{1}(S_{+}) $, we denote
\begin{equation*}
	\begin{split}
		h(f)&:=\int_{S_{+}}x_{1}f(x)dx
	\end{split}
\end{equation*}
as the horizontal impulse of $ f $. To begin with, we let $ \omg_{0} $ be in $ L^{\ift}(S_{+}) $ with a bounded support in $ S_{+} $. Here, we assume a impermeable wall $ \lbrace0\rbrace\times\bbT $, which gives a tangential boundary condition for $ u $ at $ x_{1}=0 $. Then there exists the unique weak solution $ \omg\in L^{\ift}\big(0,\ift;(L^{1}\cap L^{\ift})(S_{+})\big) $ such that the following quantities are conserved for all times; 
\begin{equation*}\label{eq_consqu}
	\begin{split}
		(a)\;\;&\text{the total mass }\;\int_{S_{+}}\omg(t,x)dx,\\
		(b)\;\;&\text{the horizontal impulse }\;h(\omg(t)),\\
		(c)\;\;&\text{the measure of a level set }\;\big|\lbrace x\in S_{+} : \omg(t,x)>\alp\rbrace\big|,\quad \forall\alp>0,\\
		(d)\;\;&\text{the }L^{p}-\text{norm }\;\big\|\omg(t)\big\|_{L^{p}(S_{+})},\quad p\in[1,\ift].
	\end{split}
\end{equation*}
Such conservation laws are needed to prove the $ J_{1}- $ stability \eqref{eq_J1estimate}.
The existence and uniqueness can be found in \cite{Kelliher15} and \cite{BeD17}. The total mass is conserved because of the divergence-free condition of $ u(t) $, and the conservation of 
the horizontal impulse is due to the odd-symmetry of the cylindrical Biot-Savart kernel $ K_{1} $, which is presented in \eqref{eq_kernelK}.
We refer the reader to \cite[Prop. 2.1]{BeD17} for details. The conservation of the measure of each level set of $ \omg(t) $ is because the flow map is measure-preserving in time. This leads to the conservation of the $ L^{p}- $norm of $ \omg(t) $, since the $ L^{p}- $norm can be represented as the integral of the measure of each level set.

\subsection{Rearrangement and its estimate}

First, we define the rearrangement of a measurable set with finite measure in $ S_{+} $ and the rearrangement of a function in $ L^{1}(S_{+}) $, analogously to \cite[Sec. 3.3]{LiLAn}.\\
\begin{defn}\label{def_rearr} For a measurable set $ \Omg\subset S_{+} $ with finite measure $ |\Omg|<\ift $, we define the rearrangement $ \Omg^{\ast} $ of $ \Omg $ as the bounded strip that has the same measure as $ \Omg $; that is,  $\Omg^{\ast}:=\lbrace x_{1}<\frac{|\Omg|}{2\pi}\rbrace
	$.
In addition, for a non-negative function $ f\in L^{1}(S_{+}) $, we define the rearrangement $ f^{\ast} $ of $ f $ as the non-negative function that satisfies
\begin{equation*}
	\lbrace f^{\ast}>\alp\rbrace=\lbrace f>\alp\rbrace^{\ast},\quad \forall\alp>0.\label{levelset}
\end{equation*}
\end{defn} By the above definition, $ f^{\ast} $ is $ x_{2}- $independent, non-increasing in $ x_{1} $, and satisfies
\begin{equation}
\big|\lbrace f^{\ast}>\alp\rbrace\big|=\big|\lbrace f>\alp\rbrace\big|,\quad \forall\alp>0.\label{levelsetmeas}
\end{equation}
From \eqref{levelsetmeas}, we have
\begin{equation}
\big\|f^{\ast}\big\|_{L^{1}(S_{+})}=\big\|f\big\|_{L^{1}(S_{+})}, \qquad h(f^{\ast})\leq h(f).\label{eq_hfast}
\end{equation} 

There are two properties of rearrangement that we use in this section. The first is the nonexpansivity, which is introduced in the following lemma.

\begin{lem}[Nonexpansivity]\label{lem_nonexp}
	Let $ f,g\in L^{1}(S_{+}) $ be non-negative functions, and let $ g  $ satisfy $ g^{\ast}=g $. Then we have
	\begin{equation*}
		\big\|f^{\ast}-g\big\|_{L^{1}(S_{+})}\leq\big\|f-g\big\|_{L^{1}(S_{+})}.\label{nonexp}
	\end{equation*}
\end{lem}
We refer the reader to \cite[Sec. 3.5]{LiLAn} for the proof of a general case where $ g^{\ast}=g $ need not hold. This is the extension of \cite[Lemma 2]{MaP85} in a bounded strip $ \lbrace x_{1}<a\rbrace,\; a>0 $. 
Recall that in Theorem \ref{thm_stabi}, we required the stationary solution $ \zt=\zt(x_{1}) $ to be non-increasing. This is because it satisfies $ \zt^{\ast}=\zt $, and this fact is used in applying the above lemma with $ g=\zt $ in the proof of Theorem \ref{thm_stabi}.

The second is the rearrangement estimate \eqref{estimate_MP}, which 
is a refinement of \eqref{eq_hfast}. Such an estimate can be found in 
\cite[Lemma 1]{MaP85} when the fluid lies on the bounded strip $ \lbrace x_{1}<a\rbrace
,\; a<\ift $. 

\begin{lem}\label{lem_MPestimate}
	For a non-negative function $ f\in L^{\ift}(S_{+}) $ satisfying $ h(f)<\ift $, we have 
	\begin{equation}
		\big\|f-f^{\ast}\big\|_{L^{1}(S_{+})}^{2}\leq C\big\|f\big\|_{L^{\ift}(S_{+})}\big[h(f)-h(f^{\ast})\big].\label{estimate_MP}
	\end{equation} with some universal constant $C>0$. 
\end{lem}
This tells us that $ f^{\ast} $ is the unique minimizer of the horizontal impulse among every function that has the same measure of each level set with $ f $. 

\begin{proof}
	Without loss of generality, we assume $ \left \|f\right \|_{L^{\ift}}=1 $. First, let us consider the case $ f=1_{\Omg} $, where $ \Omg\subset S_{+} $ is a measurable set with finite measure. 
	Note that from $ |\Omg|=|\Omg^{\ast}| $, for some $ \bt\geq0 $, we have
	\begin{equation*}
		|\Omg\setminus\Omg^{\ast}|=|\Omg^{\ast}\setminus\Omg|=\frac{1}{2}|\Omg\triangle\Omg^{\ast}|=\bt.
	\end{equation*}
	Then using $ \Omg\setminus\Omg^{\ast}=(\Omg\cup\Omg^{\ast})\setminus\Omg^{\ast} $,\; $ \Omg^{\ast}\setminus\Omg=\Omg^{\ast}\setminus(\Omg\cap\Omg^{\ast}) $, and $ h(f)\geq h(f^{\ast}) $, we obtain
	\begin{equation}
		\begin{split}
			\int_{\Omg\setminus\Omg^{\ast}}x_{1}dx\geq\int_{(\Omg\cup\Omg^{\ast})^{\ast}\setminus\Omg^{\ast}}x_{1}dx,\quad \int_{\Omg^{\ast}\setminus\Omg}x_{1}dx\leq\int_{\Omg^{\ast}\setminus(\Omg\cap\Omg^{\ast})^{\ast}}x_{1}dx.\label{ineq_setm}
		\end{split}
	\end{equation}
	Then taking $ a_{1}\geq b\geq a_{2}>0 $ that satisfy
	\begin{equation*}
		\lbrace x_{1}<a_{1}\rbrace
		=(\Omg\cup\Omg^{\ast})^{\ast},\quad \lbrace x_{1}<b\rbrace
		=\Omg^{\ast},\quad \lbrace x_{1}<a_{2}\rbrace
		=(\Omg\cap\Omg^{\ast})^{\ast},
	\end{equation*}
	we get
	\begin{equation*}
		a_{1}=b+\frac{\bt}{2\pi},\quad a_{2}=b-\frac{\bt}{2\pi}.
	\end{equation*}
	This together with the inequality \eqref{ineq_setm}, we obtain 
	\begin{equation*}
		\begin{split}
			h(1_{\Omg})-h(1_{\Omg^{\ast}})&=\int_{\Omg\setminus\Omg^{\ast}}x_{1}dx-\int_{\Omg^{\ast}\setminus\Omg}x_{1}dx\geq\int_{\lbrace b\leq x_{1}<a_{1}\rbrace
			}x_{1}dx-\int_{\lbrace a_{2}\leq x_{1}<b\rbrace
		}x_{1}dx\\
			&=2\pi\bigg(\int_{b}^{a_{1}}x_{1}dx_{1}-\int_{a_{2}}^{b}x_{1}dx_{1}\bigg)=\pi(a_{1}^{2}+a_{2}^{2}-2b^{2})=\frac{\bt^{2}}{2\pi} =\frac{1}{8\pi}\big\|1_{\Omg}-1_{\Omg^{\ast}}\big\|_{L^{1}}^{2}.
		\end{split}
	\end{equation*}
	Now we consider the case where $ f $ is a simple function of the form
	\begin{equation*}
		f(x)=\frac{1}{n}\sum_{i=1}^{n-1}1_{A_{i}}(x)
	\end{equation*}
	with some $n\ge2$, where $ \lbrace A_{i}\rbrace_{i=1}^{n-1}\subset S_{+} $ satisfies $ A_{i+1}\subset A_{i} $ for $ i=1,\cdots,n-2 $ and $ |A_{1}|<\ift $. Then its rearrangement $ f^{\ast} $ becomes
	\begin{equation*}
		f^{\ast}(x)=\frac{1}{n}\sum_{i=1}^{n-1}1_{A_{i}^{\ast}}(x).
	\end{equation*}
	Then due to the result above and the linearity of $ h $, we get
	\begin{equation*}
		\begin{split}
			h(f)-h(f^{\ast})&=\frac{1}{n}\sum_{i=1}^{n-1}\big[h(1_{A_{i}})-h(1_{A_{i}^{\ast}})\big]\geq\frac{1}{8\pi n}\sum_{i=1}^{n-1}\big\|1_{A_{i}}-1_{A_{i}^{\ast}}\big\|_{L^{1}}^{2}
		\end{split}
	\end{equation*}
	On the other hand, using the Cauchy-Schwarz inequality, we have
	\begin{equation*}
		\big\|f-f^{\ast}\big\|_{L^{1}}
		\leq\frac{1}{n}\sum_{i=1}^{n-1}\big\|1_{A_{i}}-1_{A_{i}^{\ast}}\big\|_{L^{1}}\leq\frac{\sqrt{n-1}}{n}\bigg(\sum_{i=1}^{n-1}\big\|1_{A_{i}}-1_{A_{i}^{\ast}}\big\|_{L^{1}}^{2}\bigg)^{\frac{1}{2}}.
	\end{equation*}
	This gives us
	\begin{equation}\label{est_simple}
		h(f)-h(f^{\ast})
		\geq\frac{n}{8\pi (n-1)}\big\|f-f^{\ast}\big\|_{L^{1}}^{2}.
	\end{equation}
	Finally, for general $ f $, we use the sequence of simple functions $ \lbrace g_{n}\rbrace_{n=1}^{\ift}\subset L^{\ift}(S_{+}) $ given as
	\begin{equation*}
		g_{n}(x)=\frac{1}{n}\sum_{i=1}^{n-1}1_{A_{i}^{(n)}}(x),\quad A_{i}^{(n)}=\bigg\lbrace f>\frac{i}{n}\bigg\rbrace.
	\end{equation*}
	Then $ g_{n} $ is dominated by $ f $ and it converges to $ f $ pointwise as $ n $ goes to infinity, and the same holds for $ (g_{n})^{\ast} $ and $ f^{\ast} $. Therefore, applying the estimate \eqref{est_simple} to $ g_{n} $ and using the dominated convergence theorem, the proof is complete. 
\end{proof}

\subsection{The cut-off operator and the rearrangement}

We need one more notion in order to prove Theorem \ref{thm_stabi}; the cut-off operator of a non-negative function. For any $ \alp>0 $, we define a cut-off function $ \eta_{\alp} : \bbR_{\geq0}\To\bbR_{\geq0} $ given by
\begin{equation*}
	\eta_{\alp}(s):=\begin{cases}
		s & \text{ if }s\leq \alp\\
		\alp  & \text{ if } s>\alp
	\end{cases},
\end{equation*}
and for any non-negative measurable function $ f $ on $ S_{+} $, we define the cut-off operator $ \Gmm_{\alp} $ as
\begin{equation*}
	\Gmm_{\alp}f:=\eta_{\alp}\circ f.
\end{equation*}
That is, $ \Gmm_{\alp} $ cuts off the part where $ f $ is greater than $ \alp $ and replaces it with $ \alp $, whereas the other part of $ f $ that is less than or equal to $ \alp $ stays the same. Note that if $ f $ is in $ L^{1}(S_{+}) $, where the rearrangement $ f^{\ast} $ of $ f $ is defined, then the cut-off operator and the rearrangement operator commute with one another:
\begin{equation}
	(\Gmm_{\alp}f)^{\ast}=\Gmm_{\alp}(f^{\ast}).\label{eq_commute}
\end{equation}
Indeed, it can be easily shown that $ \Gmm_{\alp}(f^{\ast}) $ satisfies the definition of $ (\Gmm_{\alp}f)^{\ast} $. For $ s\in[0,\alp] $, we have
\begin{equation*}
	\lbrace \Gmm_{\alp}(f^{\ast})>s\rbrace=\lbrace f^{\ast}>s\rbrace=\lbrace f>s\rbrace^{\ast}=\lbrace \Gmm_{\alp}f>s\rbrace^{\ast},
\end{equation*}
and for $ s\in(\alp,\ift) $, we have
\begin{equation*}
	\lbrace \Gmm_{\alp}(f^{\ast})>s\rbrace=\emptyset=\lbrace \Gmm_{\alp}f>s\rbrace^{\ast}.
\end{equation*}
This property is crucial in the proof of Theorem \ref{thm_stabi}. 

\subsection{Proof of Theorem \ref{thm_stabi}}

We finish this section by proving Theorem \ref{thm_stabi}. We shall use the notation $\Lmb_{f,\alp}:=\lbrace f>\alp\rbrace.$

\begin{proof}[Proof of Theorem \ref{thm_stabi}]
	We fix $ L,M>0 $ and let $ \zt $ be a function that satisfies the condition \eqref{eq_LM}. 
	Then note that for any non-negative function $ f\in L^{1}(S_{+}) $, the measure of the level set $ \Lmb_{f,M+1} $ can be estimated as
	\begin{equation}
		|\Lmb_{f,M+1}|=\int_{\lbrace f-M>1\rbrace}1dx\leq\int_{\lbrace f-M>1\rbrace}|f-M|dx\leq\int_{\lbrace f-M>1\rbrace}|f-\zt|dx\leq\big\|f-\zt\big\|_{L^{1}}.\label{est_lambf}
	\end{equation}
	We fix $ t\geq0 $ and for simplicity, we drop the parameter $ t $ from $ \omg(t) $ and $ M+1 $ from the level set notation and the cut-off operator;
	\begin{equation*}
		\omg=\omg(t),\quad {\Lmb_{\omg_{0}}=\Lmb_{\omg_{0},M+1}},\quad \Lmb_{\omg}=\Lmb_{\omg(t),M+1},\quad 
		\Gmm\omg=\Gmm_{M+1}[\omg(t)].
	\end{equation*}
	To begin with, we estimate $ h(|\omg-\zt|) $ using the decomposition
	\begin{equation*}
		\int_{\lbrace x_{1}\geq L\rbrace
		}x_{1}\omg dx=\big[ h(\omg)-h(\zeta) \big]-\int_{\lbrace x_{1}<L\rbrace
		}x_{1} (\omg-\zt)dx.
	\end{equation*}
	Then using the non-negativity of $ \omg $, we can use this decomposition to obtain
	\begin{equation*}
		\begin{split}
			h(|\omg-\zt|)&=\int_{\lbrace x_{1}<L\rbrace}x_{1}| \omg-\zeta| dx+\int_{\lbrace x_{1}\geq L\rbrace}x_{1}\underbrace{|\omg- \zeta|}_{=|\omg|=\omg} dx \leq L\int_{\lbrace x_{1}<L\rbrace}| \omg-\zeta| dx+\int_{\lbrace x_{1}\geq L\rbrace}x_{1}\omg dx\\
			&\leq L\int_{\lbrace x_{1}<L\rbrace}| \omg-\zeta| dx+ \big[h(\omg_{0})-h(\zeta)\big]+\int_{\lbrace x_{1}<L\rbrace}x_{1}  \big|\omg-\zeta \big|dx.
		\end{split}
	\end{equation*}
	In the above, the conservation $ h(\omg)=h(\omg_{0}) $ 
	was used. Then we have
	\begin{equation*}
		\begin{split}
			h(|\omg-\zt|)&\leq2L\int_{\lbrace x_{1}<L\rbrace}| \omg-\zeta| dx+h(|\omg_{0}-\zt|)\leq2L\big\|\omg-\zt\big\|_{L^{1}}+h(|\omg_{0}-\zt|).
		\end{split}
	\end{equation*}
	This shows us that it is only left to estimate $ \big\|\omg-\zt\big\|_{L^{1}} $ to obtain an estimate of $ \big\|\omg-\zt\big\|_{J_{1}} $. First, we split the integral of $ \big\|\omg-\zt\big\|_{L^{1}} $ to get
	\begin{equation*}
		\big\|\omg-\zt\big\|_{L^{1}}=\int_{\Lmb_{\omg}}| \omg-\zt| dx+\int_{\Lmb_{\omg}^{C}}| \Gmm\omg-\zt| dx.
	\end{equation*}
	In the right-hand side, we can estimate the first term:
	\begin{equation*}
		\begin{split}
			\int_{\Lmb_{\omg}}| \omg-\zt| dx&\leq\int_{\Lmb_{\omg}}\omg dx+\int_{\Lmb_{\omg}}\zt dx=\int_{\Lmb_{\omg_{0}}}\omg_{0} dx+\int_{\Lmb_{\omg}}\zt dx\leq\int_{\Lmb_{\omg_{0}}}|\omg_{0}-\zt| dx+\int_{\Lmb_{\omg_{0}}}\zt dx+M\cdot\underbrace{|\Lmb_{\omg}|}_{=|\Lmb_{\omg_{0}}|}\\
			&\leq\big\|\omg_{0}-\zt\big\|_{L^{1}}+M\cdot|\Lmb_{\omg_{0}}|+M\cdot|\Lmb_{\omg_{0}}|\leq(2M+1)\big\|\omg_{0}-\zt\big\|_{L^{1}}.
		\end{split}
	\end{equation*}
	The third inequality follows by using $ \zt\leq M $ and the conservation $ |\Lmb_{\omg}|=|\Lmb_{\omg_{0}}| $. The last inequality comes from the estimate \eqref{est_lambf} of $ |\Lmb_{\omg_{0}}| $.\\
	The second term is estimated by using the estimate \eqref{estimate_MP} from Lemma \ref{lem_MPestimate} and the nonexpansivity from Lemma \ref{lem_nonexp}, due to $ (\Gmm\omg)^{\ast}=(\Gmm\omg_{0})^{\ast} $ and $ \zt^{\ast}=\zt $, respectively. We also use the commutative property $ (\Gmm\omg_{0})^{\ast}=\Gmm[(\omg_{0})^{\ast}] $ from \eqref{eq_commute}:
	\begin{equation*}
		\begin{split}
			\int_{\Lmb_{\omg}^{C}}| \Gmm\omg-\zt| dx&\leq\big\|\Gmm\omg-\zt\big\|_{L^{1}}\leq\big\|\Gmm\omg-(\Gmm\omg)^{\ast}\big\|_{L^{1}}+\big\|(\Gmm\omg_{0})^{\ast}-\zt\big\|_{L^{1}}\\
			&\leq\sqrt{C(M+1)}\big[h(\Gmm\omg)-h\big((\Gmm\omg_{0})^{\ast}\big)\big]^{\frac{1}{2}}+\big\|\Gmm\omg_{0}-\zt\big\|_{L^{1}}\\
			&\leq\sqrt{C(M+1)}\big[h(\Gmm\omg)-h\big(\Gmm[(\omg_{0})^{\ast}]\big)\big]^{\frac{1}{2}}+\big\|\omg_{0}-\zt\big\|_{L^{1}}.
		\end{split}
	\end{equation*}
	Furthermore, the term $ h(\Gmm\omg)-h\big(\Gmm[(\omg_{0})^{\ast}]\big) $ from the above can be estimated by adding and subtracting suitable terms, using $ \Gmm\omg\leq\omg $, and using the conservation $ h(\omg)=h(\omg_{0}) $:
	\begin{equation*}
		\begin{split}
			h(\Gmm\omg
			)-h\big(\Gmm[(\omg_{0})^{\ast}]
			\big)=&\;\big[h(\Gmm\omg
			)-h(\omg)\big]
			+\big[h(\omg_{0})
			-h(\zt)\big]+\big[h(\zt)-h\big((\omg_{0})^{\ast}\big)\big]+\big[h\big((\omg_{0})^{\ast}\big)-h\big(\Gmm[(\omg_{0})^{\ast}]
			\big)\big]\\
			\leq&\;h(|\omg_{0}-\zt|)+h\big(\big|(\omg_{0})^{\ast}-\zt\big|\big)
			+\int_{\Lmb_{(\omg_{0})^{\ast},M+1}}x_{1}(\omg_{0})^{\ast}dx.
		\end{split}
	\end{equation*}
	The second term in the right-hand side of the above inequality 
	can be estimated in the same way as we did for the term $ h(|\omg-\zt|) $, which gives us
	\begin{equation*}
		h\big(\big|(\omg_{0})^{\ast}-\zt\big|\big)\leq2L\big\|(\omg_{0})^{\ast}-\zt\big\|_{L^{1}}+h(|\omg_{0}-\zt|)\leq2L\big\|\omg_{0}-\zt\big\|_{L^{1}}+h(|\omg_{0}-\zt|).
	\end{equation*}
	For the third term, 
	note that $ \Lmb_{(\omg_{0})^{\ast},M+1}=\lbrace x_{1}<a_{0}\rbrace
	 $ for some $ a_{0}\geq0 $, and from the level set measure conservation of rearrangement and the estimate \eqref{est_lambf} of $ |\Lmb_{\omg_{0}}| $, we have
	\begin{equation*}
		a_{0}=\frac{|\Lmb_{(\omg_{0})^{\ast},M+1}|}{2\pi}=\frac{|\Lmb_{\omg_{0}}|}{2\pi}\leq\frac{1}{2\pi}\big\|\omg_{0}-\zt\big\|_{L^{1}}.
	\end{equation*}
	Using this and the $ L^{1}- $norm conservation of rearrangement, we get
	\begin{equation*}
		\begin{split}
			\int_{\Lmb_{(\omg_{0})^{\ast},M+1}}x_{1}(\omg_{0})^{\ast}dx&\leq a_{0}\int_{\lbrace x_{1}<a_{0}\rbrace
			}(\omg_{0})^{\ast} dx\leq\frac{1}{2\pi}\big\|\omg_{0}-\zt\big\|_{L^{1}}\big\|(\omg_{0})^{\ast}\big\|_{L^{1}}=\frac{1}{2\pi}\big\|\omg_{0}-\zt\big\|_{L^{1}}\big\|\omg_{0}\big\|_{L^{1}}\\
			&\leq\frac{1}{2\pi}\big\|\omg_{0}-\zt\big\|_{L^{1}}^{2}+\frac{1}{2\pi}\big\|\zt\big\|_{L^{1}}\big\|\omg_{0}-\zt\big\|_{L^{1}}\leq\frac{1}{2\pi}\big\|\omg_{0}-\zt\big\|_{L^{1}}^{2}+LM\big\|\omg_{0}-\zt\big\|_{L^{1}}.
		\end{split}
	\end{equation*}
	Thus, gathering all the above estimates, we have
	\begin{equation*}
		\begin{split}
			\big\|\omg-\zt\big\|_{L^{1}}\leq&\;C_{L,M}\Big[\big\|\omg_{0}-\zt\big\|_{L^{1}}^{\frac{1}{2}}+\big\|\omg_{0}-\zt\big\|_{L^{1}}+h(|\omg_{0}-\zt|)^{\frac{1}{2}}\Big],
		\end{split}
	\end{equation*}
	where $ C_{L,M}>0 $ is a constant that depends only on $ L,M $. Finally, we obtain
	\begin{equation*}
		\begin{split}
			\big\|\omg-\zt\big\|_{J_{1}}&\leq(2L+1)\big\|\omg-\zt\big\|_{L^{1}}+h(|\omg_{0}-\zt|) \leq C_{1}\Big[\big\|\omg_{0}-\zt\big\|_{J_{1}}^{\frac{1}{2}}+\big\|\omg_{0}-\zt\big\|_{J_{1}}\Big],
		\end{split}
	\end{equation*}
	where $ C_{1}=C_{1}(L,M)>0 $ is a constant.
\end{proof}

\section{Perimeter growth}

\subsection{Velocity and the total mass of the stationary patch $ 1_{\br{\Omg}} $}

We consider the patch-type vorticity $\br{\omg}=1_{\br{\Omg}}$ on $ S_{+} $. It defines a stationary solution because 
$\br{\omg}=\br{\omg}(x_1)$
is $ x_{2}- $independent. The velocity field $ \br{u} $ on $ S_{+} $ that corresponds to $ \br\omg $ by the Biot--Savart law \eqref{BiotSavart} can be calculated, using the stream function $ \Psi =\Psi(x_1)$.
We denote $ \Psi' $ as the first derivative of $ \Psi $ with respect to the $ x_{1}- $variable. Then using the conditions from the problem \eqref{eq_elliptic}, we obtain
\begin{equation*}
	\begin{split}
		\br{u}_{2}(x)&=\Psi'(x_{1})=-\Psi'(\xi)\bigg|_{\xi=x_{1}}^{\xi=\ift}=-\int_{x_{1}}^{\ift}\underbrace{\Psi''(\xi)}_{=\lap\Psi(\xi)}d\xi= \int_{x_{1}}^{\ift}1_{(0,1)}(\xi)d\xi=\begin{cases}
			1-x_{1} & \text{if }0\leq x_{1}<1,\\
			0 & \text{if }x_{1}\geq1
		\end{cases}.
	\end{split}
\end{equation*}
In addition, we can calculate the total mass $ m(1_{\br{\Omg}}) $ 
of $ 1_{\br{\Omg}} $:
\begin{equation*}
	\begin{split}
		m(1_{\br{\Omg}})&=\int_{S_{+}}1_{\br{\Omg}}(x)dx=\int_{\br{\Omg}}1\ dx
		=\int_{0}^{1}\int_{-\pi}^{\pi}1\ dx_{2}dx_{1}=2\pi.
	\end{split}
\end{equation*}

\subsection{Framework of understanding the patch $ 1_{\Omg_{t}} $ on $ S_{+} $
}

To begin with, we let $ \Pi : \bbR \To \bbT $ be the quotient map $ \Pi(x):=x-2n\pi $ for some $ n\in\bbZ $ that satisfies $ x\in[(2n-1)\pi,(2n+1)\pi) $. Additionally, we define a map $ Q : \bbR_{+}^{2}\To S_{+} $ by
\begin{equation*}
	Q(x_{1},x_{2}):=(x_{1},\Pi(x_{2}))=(x_{1},x_{2}-2n\pi),
\end{equation*}
for the same $ n\in\bbZ $ above.

We let $ \omg_{0}=1_{\Omg_{0}} $ with $ \Omg_{0}\subset S_{+} $. Then we obtain the unique solution $ \omg(t)=1_{\Omg_{t}} $ of \eqref{Eulereqs} on $ S_{+} $ with the initial data $ \omg_{0}=1_{\Omg_{0}} $, and get the velocity field $ u(t) $ on $ S_{+} $ by the cylindrical Biot--Savart law \eqref{BiotSavart}. We introduce one way of understanding the patch $ 1_{\Omg_{t}} $.
\noindent Let us consider the periodic extension $ u_{ext}(t) $ on $ \bbR_{+}^{2} $ of $ u(t) $, defined as
\begin{equation}\label{eq_uext}
	u_{ext}(t,x):=u\big(t,Q(x)\big),\quad t\geq0,\quad x\in\bbR_{+}^{2}.
\end{equation}
Then if we let $ t\geq0 $ and $ x\in\bbR_{+}^{2} $, then $ u_{ext}(t) $ 
induces the flow map $ \Phi(t) $ on $ \bbR_{+}^{2} $ that satisfies the ODE \eqref{eq_flowmapext}. 
Note that $ \Phi(t) $ is well-defined because $ u_{ext}(t) $ is uniformly bounded in time and has the log-Lipschitz estimate:
\begin{equation*}
	\big\|u_{ext}(t)\big\|_{L^{\ift}(\bbR_{+}^{2})}\leq C\big\|\omg_{0}\big\|_{(L^{1}\cap L^{\ift})(S_{+})},\quad t\geq0,\label{eq_uunifbd}
\end{equation*}
\begin{equation*}
	|u_{ext}(t,x)-u_{ext}(t,z)|\leq C\big\|\omg_{0}\big\|_{(L^{1}\cap L^{\ift})(S_{+})}|x-z|(1-\ln|x-z|),\quad t\geq0,\quad|x-z|\leq1.\label{eq_logLip}
\end{equation*}
See \cite[Sec. 2.3]{MaPMT}, especially \cite[Lemma 3.2]{MaPMT}, or \cite[Sec. 8.2.3]{MaBVo}. 
Then we consider $ \Phi(t,\Omg_{0})\subset\bbR_{+}^{2} $, and 
$ \Omg_{t}\subset S_{+}
 $ is the image of $ \Phi(t,\Omg_{0}) $ through the map $ Q $:
\begin{equation*}
	\Omg_{t}
	=Q\big(\Phi(t,\Omg_{0})\big)=Q\big(\Phi_{t}(\Omg_{0})\big).
\end{equation*}
That is, $ \Omg_{t} $ is the union of intersection between $ \Phi_{t}(\Omg_{0}) $ and $ \lbrace (x_{1},x_{2})\in\bbR_{+}^{2} : (2k-1)\pi\leq x_{2}<(2k+1)\pi\rbrace $ that are translated by the distance $ 2k\pi $ in $ x_{2} $ from $ k=-N $ to $ k=N $:
\begin{equation}\label{eq_Omgt}
	\begin{split}
		\Omg_{t}&=\bigcup_{k=-N}^{N}\Omg_{t,k},\quad \Omg_{t,k}:=E_{t,k}-\begin{pmatrix}
			0\\
			2k\pi
		\end{pmatrix},\\
		E_{t,k}&:=\Phi_{t}(\Omg_{0})\cap\lbrace (x_{1},x_{2})\in\bbR_{+}^{2} : (2k-1)\pi\leq x_{2}<(2k+1)\pi\rbrace,
	\end{split}
\end{equation}
where $ N=N(t)\in\bbN $ satisfies $ \Phi_{t}(\Omg_{0})\subset\bbR_{>0}\times[-(2N+1)\pi,(2N+1)\pi) $.

\subsection{Growth rate of the vertical center of mass of $ 1_{\br{\Phi}_{t}(\br{\Omg})} $
}

For a function $ f $ defined on $ \bbR_{+}^{2} $ in which both $ f $ and $ x_{2}f $ are in $ L^{1}(\bbR_{+}^{2}) $, we define 
the vertical center of mass $ k(f) $ as
\begin{equation*}
	\begin{split}
		k(f):=\bigg(\int_{\bbR_{+}^{2}}f(x)dx\bigg)^{-1}\cdot
		\int_{\bbR_{+}^{2}}x_{2}f(x)dx.
	\end{split}
\end{equation*}
As in \eqref{eq_uext}, we define $ \br{u}_{ext} $, the periodic extension of $ \br{u} $, which is the steady velocity from $ \br{\omg}=1_{\br{\Omg}} $, by
\begin{equation*}
	\br{u}_{ext}(x):=\br{u}\big(Q(x)\big),\quad x\in\bbR_{+}^{2}.
\end{equation*}
We consider $ \br{\Phi}(t) $, the flow map from $ \br{u}_{ext} $, and the image of $ \br{\Omg} $ through $ \br{\Phi}(t) $:
$$ {\br{\Phi}_{t}(\br{\Omg})=\br{\Phi}(t,\br{\Omg})=\lbrace (x_{1},x_{2})\in\bbR_{+}^{2} : 0<x_{1}<1,\; (1-x_{1})t-\pi\leq x_{2}<(1-x_{1})t+\pi\rbrace.} $$
Then we can calculate $ k(1_{\br{\Omg}}) $;
\begin{equation*}
	\begin{split}
		k(1_{\br{\Omg}})&
		=\bigg(\int_{\br{\Omg}}1\;dx\bigg)^{-1}\cdot
		\int_{\br{\Omg}}x_{2}dx =\frac{1}{2\pi}\int_{0}^{1}\int_{-\pi}^{\pi}x_{2}dx_{2}dx_{1}=0,
	\end{split}
\end{equation*}
and $ \frac{d}{dt}k(1_{\br{\Phi}_{t}(\br{\Omg})}) $, the growth rate of $ k(1_{\br{\Phi}_{t}(\br{\Omg})}) $:
\begin{equation}
	\begin{split}
		\frac{d}{dt}k(1_{\br{\Phi}_{t}(\br{\Omg})})&
		=
		\frac{d}{dt}\bigg[\bigg(\int_{\br{\Phi}_{t}(\br{\Omg})}1\ dx\bigg)^{-1}\cdot\int_{\br{\Phi}_{t}(\br{\Omg})}x_{2}dx\bigg]=\bigg(\int_{\br{\Omg}}1\ dx\bigg)^{-1}\cdot\frac{d}{dt}\int_{\br{\Phi}_{t}(\br{\Omg})}x_{2}dx\\
		&=\frac{1}{2\pi}\frac{d}{dt}\int_{\br{\Omg}}\br{\Phi}_{2}(t,x)dx=\frac{1}{2\pi}\int_{\br{\Omg}}\rd_{t}\br{\Phi}_{2}(t,x)dx=\frac{1}{2\pi}\int_{\br{\Omg}}\br{u}_{ext,2}(x)dx\\
		&
		=\frac{1}{2\pi}\int_{\br{\Omg}}\br{u}_{2}(x)dx=\frac{1}{2\pi}\int_{-\pi}^{\pi}\int_{0}^{1}(1-x_{1})dx_{1}dx_{2}=\frac{1}{2}.
		\label{eq_derivofRt'1}
	\end{split}
\end{equation}

\subsection{Growth rate of the vertical center of mass of $ 1_{\Phi_{t}(\Omg_{0})} $}

In the following lemma, we show that the difference between the vertical velocities $ u_{2}(t) $ of $ 1_{\Omg_{t}} $ and $ \br{u}_{2} $ of $ 1_{\br{\Omg}} $ is controlled by the $ L^{1}- $difference between the corresponding patches.

\begin{lem}\label{lem_uift}
	There exists a constant $ C_{3}>0 $ such that if $ \Omg_{0}\subset S_{+} $ is a bounded open set, then for any $ t\geq0 $ and $ x\in S_{+} $, the velocity fields $ u(t) $ and $ \br{u} $ that are determined by vortex patches $ \omg(t)=1_{\Omg_{t}} $ and $ \br{\omg}=1_{\br{\Omg}} $, respectively, by the cylindrical Biot-Savart law \eqref{BiotSavart} satisfy
	\begin{equation}\label{eq_uift}
		|u_{2}(t,x)-\br{u}_{2}(x)|\leq C_{3}\big[|\Omg_{t}\triangle \br{\Omg}|^{\frac{1}{2}}+|\Omg_{t}\triangle \br{\Omg}|\big].
	\end{equation}
\end{lem}
To prove this lemma, we use an elementary lemma {(e.g. see Iftimie--Sideris--Gamblin \cite{ISG})} that is used in bounding the velocity term by the $ L^{1}- $norm and the $ L^{\ift}- $norm of the vorticity term.
\begin{lem}[{{\cite[Lemma 2.1]{ISG}}}]\label{lem_ISG}
	Let $ A\subset \bbR^{2} $ and $ f\in (L^{1}\cap L^{\ift})(A) $ be a non-negative function. Then there exists an absolute constant $ C'>0 $ such that for any $ x\in\bbR^{2} $, we have
	\begin{equation*}\label{eq_ISG}
		\int_{A}\frac{f(y)}{|x-y|}dy\leq C'\big\|f\big\|_{L^{1}(A)}^{\frac{1}{2}}\big\|f\big\|_{L^{\ift}(A)}^{\frac{1}{2}}.
	\end{equation*}
\end{lem}
\begin{proof}[Proof of Lemma \ref{lem_uift}]
	We fix $ x=(x_{1},x_{2})\in S_{+} $. First, we define the periodic extensions $ \Omg_{t,ext}, \br{\Omg}_{ext}\subset\bbR_{+}^{2} $ of sets $ \Omg_{t}, \br{\Omg}\subset S_{+} $ as
	\begin{equation*}
		\begin{split}
			\Omg_{t,ext}=\Omg_{t},\quad& \br{\Omg}_{ext}=\br{\Omg}\quad\text{in}\quad S_{+},\\
			\Omg_{t,ext}+\begin{pmatrix}
				0\\2k\pi
			\end{pmatrix}=\Omg_{t,ext},\quad& \br{\Omg}_{ext}+\begin{pmatrix}
			0\\2k\pi
		\end{pmatrix}=\br{\Omg}_{ext}\quad\text{for any}\quad k\in\bbZ.
		\end{split}
	\end{equation*}
	Also, we define a subset $ A\subset\bbR_{+}^{2} $ as $ A:=\bbR_{>0}\times[-\pi+x_{2},\pi+x_{2}) $. Then using that $ K_{2}(x,\cdot) $, $ 1_{\Omg_{t,ext}} $, and $ 1_{\br{\Omg}_{ext}} $ are $ 2\pi- $periodic in $ y_{2} $, we have
	\begin{equation*}
		\begin{split}
			|u_{2}(t,x)-\br{u}_{2}(x)|&\leq\int_{S_{+}}|K_{2}(x,y)|\cdot|1_{\Omg_{t}}(y)-1_{\br{\Omg}}(y)|dy=\int_{A}|K_{2}(x,y)|\cdot|1_{\Omg_{t,ext}}(y)-1_{\br{\Omg}_{ext}}(y)|dy.
		\end{split}
	\end{equation*}
	Now we use the estimate \eqref{eq_kernelsest} of $ K_{2} $ and the above lemma to get
	\begin{equation*}
		\begin{split}
			|u_{2}(t,x)-\br{u}_{2}(x)|&\leq\int_{A}|K_{2}(x,y)|\cdot|1_{\Omg_{t,ext}}(y)-1_{\br{\Omg}_{ext}}(y)|dy\leq\int_{A}\bigg(\frac{C_{0}}{|x-y|}+1\bigg)\cdot|1_{\Omg_{t,ext}}(y)-1_{\br{\Omg}_{ext}}(y)|dy\\
			&\leq C_{0}'\big\|1_{\Omg_{t,ext}}-1_{\br{\Omg}_{ext}}\big\|_{L^{1}(A)}^{\frac{1}{2}}\big\|1_{\Omg_{t,ext}}-1_{\br{\Omg}_{ext}}\big\|_{L^{\ift}(A)}^{\frac{1}{2}}+\big\|1_{\Omg_{t,ext}}-1_{\br{\Omg}_{ext}}\big\|_{L^{1}(A)}\\
			&=C_{0}'\big\|1_{\Omg_{t}}-1_{\br{\Omg}}\big\|_{L^{1}(S_{+})}^{\frac{1}{2}}+\big\|1_{\Omg_{t}}-1_{\br{\Omg}}\big\|_{L^{1}(S_{+})},
		\end{split}
	\end{equation*}
	for some $ C_{0}'>0 $. Finally, taking $ C_{3}:=\max\lbrace1, C_{0}'\rbrace $, we have
	\begin{equation*}
		|u_{2}(t,x)-\br{u}_{2}(x)|\leq C_{0}'\big\|1_{\Omg_{t}}-1_{\br{\Omg}}\big\|_{L^{1}(S_{+})}^{\frac{1}{2}}+\big\|1_{\Omg_{t}}-1_{\br{\Omg}}\big\|_{L^{1}(S_{+})}\leq C_{3}\big[|\Omg_{t}\triangle\br{\Omg}|^{˝\frac{1}{2}}+|\Omg_{t}\triangle\br{\Omg}|\big].
	\end{equation*}
\end{proof}
As a corollary of this lemma, 
applying the $ J_{1}- $stability, obtained in Theorem \ref{thm_stabi}, on \eqref{eq_uift} allows us to control the difference between $ u_{2}(t) $ and $ \br{u}_{2} $ with the difference between $ \Omg_{0} $ and $ \br{\Omg} $ in measure.
\begin{cor}\label{cor_uift}
	There exists a constant $ C_{4}>0 $ such that if $ \Omg_{0}\subset \lbrace x_{1}<3\rbrace $ is a bounded open set that satisfies $ |\Omg_{0}\triangle \br{\Omg}|\leq1 $, then for any $ t\geq0 $ and $ x\in S_{+} $, the velocity fields $ u(t) $ and $ \br{u} $ from the previous lemma satisfy
	\begin{equation}\label{eq_u1u2bd}
		|u_{2}(t,x)-\br{u}_{2}(x)|\leq C_{4}|\Omg_{0}\triangle \br{\Omg}|^{\frac{1}{4}}.
	\end{equation}
\end{cor}

\begin{proof}
	Using \eqref{eq_J1estimate} from Theorem \ref{thm_stabi} with $ L=M=1 $, $ \zt=1_{\br{\Omg}} $, and $ \omg(t)=1_{\Omg_{t}} $, we can take an absolute constant $ C_{1}>0 $ that satisfies
	\begin{equation*}
		\begin{split}
			|\Omg_{t}\triangle \br{\Omg}|&=\int_{\Omg_{t}\triangle \br{\Omg}}1dx\leq\int_{\Omg_{t}\triangle \br{\Omg}}(1+x_{1})dx\leq C_{1}\bigg[\bigg(\int_{\Omg_{0}\triangle \br{\Omg}}(1+x_{1})dx\bigg)^{\frac{1}{2}}+\int_{\Omg_{0}\triangle \br{\Omg}}(1+x_{1})dx\bigg].
		\end{split}
	\end{equation*}
	Then using the conditions $ \Omg_{0}\subset \lbrace x_{1}<3\rbrace $ and $ |\Omg_{0}\triangle \br{\Omg}|\leq1 $, we have
	\begin{equation}
		\begin{split}
			|\Omg_{t}\triangle \br{\Omg}|&\leq C_{1}\bigg[\bigg(\int_{\Omg_{0}\triangle \br{\Omg}}(1+x_{1})dx\bigg)^{\frac{1}{2}}+\int_{\Omg_{0}\triangle \br{\Omg}}(1+x_{1})dx\bigg]\\
			&\leq 4C_{1}\big[|\Omg_{0}\triangle \br{\Omg}|^{\frac{1}{2}}+|\Omg_{0}\triangle \br{\Omg}|\big]\leq 8C_{1}|\Omg_{0}\triangle \br{\Omg}|^{\frac{1}{2}}.\label{eq_stabi6C1}
		\end{split}
	\end{equation}
	Finally for any $ t\geq0 $ and $ x\in S_{+} $, combining this with \eqref{eq_uift} from the previous lemma, we obtain
	\begin{equation*}
		\begin{split}
			|u_{2}(t,x)-\br{u}_{2}(x)|&\leq C_{3}\big[|\Omg_{t}\triangle \br{\Omg}|^{\frac{1}{2}}+|\Omg_{t}\triangle \br{\Omg}|\big] \leq C_{3}\big[\sqrt{8C_{1}}|\Omg_{0}\triangle \br{\Omg}|^{\frac{1}{4}}+8C_{1}|\Omg_{0}\triangle \br{\Omg}|^{\frac{1}{2}}\big]\leq C_{4}|\Omg_{0}\triangle \br{\Omg}|^{\frac{1}{4}}
		\end{split}
	\end{equation*}
	with $ C_{4}:=C_{3}(\sqrt{8C_{1}}+8C_{1}) $.
\end{proof}

Using this corollary, we prove that the growth rate of the vertical center of mass $ \frac{d}{dt}k(1_{\Phi_{t}(\Omg_{0})}) $ of $ 1_{\Phi_{t}(\Omg_{0})} $ is close to $ \frac{d}{dt}k(1_{\br{\Phi}_{t}(\br{\Omg})}) $ in \eqref{eq_derivofRt'1} if the sets $ \Omg_{0} $ and $ \br{\Omg} $ in $ S_{+} $ are close with each other in measure.
\begin{lem}\label{lem_comgrowth}
	There exists a constant $ C_{5}>0 $ such that if $ \Omg_{0}\subset \lbrace x_{1}<3\rbrace $ from the previous corollary satisfy $ |\Omg_{0}|=|\br{\Omg}| $ as well, then for any $ t\geq0 $, we have
	\begin{equation}\label{eq_comgrowth}
		\bigg|\frac{d}{dt}k(1_{\Phi_{t}(\Omg_{0})})-\frac{d}{dt}k(1_{\br{\Phi}_{t}(\br{\Omg})})\bigg|\leq C_{5}|\Omg_{0}\triangle \br{\Omg}|^{\frac{1}{4}}.
	\end{equation}
\end{lem}

\begin{proof}
	We fix $ t\geq0 $. Then following a similar procedure as in \eqref{eq_derivofRt'1}, we have
	\begin{equation*}\label{eq_comofOmgt}
		\begin{split}
			\frac{d}{dt}k(1_{\Phi_{t}(\Omg_{0})})&=
			\frac{d}{dt}\bigg[\bigg(\int_{\Phi_{t}(\Omg_{0})}1\ dx\bigg)^{-1}\cdot\int_{\Phi_{t}(\Omg_{0})}x_{2}dx\bigg]=
			\bigg(\int_{\Omg_{0}}1\ dx\bigg)^{-1}\cdot\int_{\Omg_{0}}\frac{d}{dt}\Phi_{2}(t,x)dx\\
			&=\frac{1}{2\pi}\int_{\Omg_{0}}u_{ext,2}\big(t,\Phi(t,x)\big)dx=\frac{1}{2\pi}\int_{\Phi_{t}(\Omg_{0})}u_{ext,2}(t,x)dx.
		\end{split}
	\end{equation*}
	Then taking $ N\in\bbN $ that satisfies $ \Phi_{t}(\Omg_{0})\subset\bbR_{>0}\times[-(2N+1)\pi,(2N+1)\pi) $ and using notations
	$$ E_{t,k}=\Phi_{t}(\Omg_{0})\cap\lbrace(2k-1)\pi\leq x_{2}<(2k+1)\pi\rbrace,\quad\Omg_{t,k}=E_{t,k}-\begin{pmatrix}
		0\\
		2k\pi
	\end{pmatrix},\quad k\in\lbrace-N,\cdots N\rbrace, $$
	from \eqref{eq_Omgt}, we have
	\begin{equation*}
		\begin{split}
			\frac{d}{dt}k(1_{\Phi_{t}(\Omg_{0})})&=\frac{1}{2\pi}\int_{\Phi_{t}(\Omg_{0})}u_{ext,2}(t,x)dx=\frac{1}{2\pi}\sum_{k=-N}^{N}\int_{E_{t,k}}u_{2}\big(t,Q(x)\big)dx\\
			&=\frac{1}{2\pi}\sum_{k=-N}^{N}\int_{E_{t,k}}u_{2}\big(t,x-(0,2k\pi)\big)dx=\frac{1}{2\pi}\sum_{k=-N}^{N}\int_{\Omg_{t,k}}u_{2}(t,x)dx
			=\frac{1}{2\pi}\int_{\Omg_{t}}u_{2}(t,x)dx.
		\end{split}
	\end{equation*}
	Now note that from the inequality \eqref{eq_stabi6C1}, we can get
	\begin{equation*}\label{eq_intofbruonOmgt}
		\begin{split}
			\bigg|\int_{\Omg_{t}}\br{u}_{2}(x)dx-\int_{\br{\Omg}}\br{u}_{2}(x)dx\bigg|&=\bigg|\int_{\Omg_{t}\setminus \br{\Omg}}\br{u}_{2}(x)dx-\int_{\br{\Omg}\setminus \Omg_{t}}\br{u}_{2}(x)dx\bigg|\leq\int_{\Omg_{t}\triangle \br{\Omg}}|\br{u}_{2}(x)|dx\\
			&\leq\big\|\br{u}_{2}\big\|_{L^{\ift}(S_{+})}\cdot|\Omg_{t}\triangle \br{\Omg}|\leq 8C_{1}|\Omg_{0}\triangle \br{\Omg}|^{\frac{1}{2}}.
		\end{split}
	\end{equation*}
	Using this and \eqref{eq_u1u2bd} from Corollary \ref{cor_uift}, we can obtain upper and lower bounds of $ \frac{d}{dt}k(1_{\Phi_{t}(\Omg_{0})}) $. For the upper bound, we have
	\begin{equation*}\label{eq_upperbdofcom}
		\begin{split}
			\frac{d}{dt}k(1_{\Phi_{t}(\Omg_{0})})&=\frac{1}{2\pi}\int_{\Omg_{t}}u_{2}(t,y)dy\leq\frac{1}{2\pi}\int_{\Omg_{t}}\big[\br{u}_{2}(x)+C_{4}|\Omg_{0}\triangle \br{\Omg}|^{\frac{1}{4}}\big]dx\\
			&\leq\frac{1}{2\pi}\bigg[\int_{\br{\Omg}}\br{u}_{2}(x)dx+8C_{1}|\Omg_{0}\triangle \br{\Omg}|^{\frac{1}{2}}+C_{4}|\Omg_{t}|\cdot|\Omg_{0}\triangle \br{\Omg}|^{\frac{1}{4}}\bigg]\\
			&\leq\frac{d}{dt}k(1_{\br{\Phi}_{t}(\br{\Omg})})+(2C_{1}+C_{4})|\Omg_{0}\triangle \br{\Omg}|^{\frac{1}{4}}.
		\end{split}
	\end{equation*}
	In the same way, the following lower bound can be derived: 
	\begin{equation*}\label{eq_lowerbdofcom}
		\begin{split}
			\frac{d}{dt}k(1_{\Phi_{t}(\Omg_{0})})\geq\frac{d}{dt}k(1_{\br{\Phi}_{t}(\br{\Omg})})-(2C_{1}+C_{4})|\Omg_{0}\triangle \br{\Omg}|^{\frac{1}{4}}.
		\end{split}
	\end{equation*}
	Finally, taking $ C_{5}:=2C_{1}+C_{4} $, we obtain \eqref{eq_comgrowth}.
\end{proof}

\subsection{Proof of Theorem \ref{thm_perim}}

We finish this paper by proving Theorem \ref{thm_perim}. 

\begin{proof}[Proof of Theorem \ref{thm_perim}]
	Let $ \Omg_{0}\subset S_{+} $ be a bounded, open 
	set with a smooth 
	boundary $ \rd\Omg_{0} $ that satisfies the following conditions (see Figure \ref{fig:patch}):
	\begin{equation}\label{eq_condiofOmg}
		\begin{split}
			(a)&{\;\;\text{there exists a closed curve}\;\gmm : [0,1]\To \br{S_{+}}\;\text{such that}}
			\quad {			 			 \gmm([0,1])=\rd\Omg_{0}\;\text{and}\;\gmm(0)=\gmm(1);}\\
			(b)&\;\; \text{length}(\rd\Omg_{0})\leq20;\\
			(c)&\;\;|\Omg_{0}|=|\br{\Omg}|;\\
			(d)&\;\;\dlt:=|\Omg_{0}\triangle \br{\Omg}|\leq\min\bigg\lbrace1,\bigg(\frac{1}{4(C_{4}+C_{5})}\bigg)^{4}\bigg\rbrace;\\
			(e)&\;\;k(1_{\Omg_{0}})=0;\\
			(f)&\;\;
			(0,0)\in\rd\Omg_{0}\cap \rd S_{+}.
		\end{split}
	\end{equation}
	{Note that the set $ \br{\Omg} $ does not satisfy the condition (a) due to disconnectedness
of its boundary $ \rd\br{\Omg}=(\lbrace0\rbrace\times\bbT)\cup (\lbrace 1\rbrace\times\bbT)$.} We fix $ t\geq0 $. We claim that there exist two points $ a=(a_{1},a_{2}),\ b=(b_{1},b_{2})\in\rd\Phi_{t}(\Omg_{0}) $ depending on $ t $ and an absolute constant $ C_{2}>0 $ such that they satisfy
	\begin{equation}\label{eq_claimthm2}
		a_{2}-C_{2}t\geq k(1_{\Phi_{t}(\Omg_{0})})\geq b_{2}.
	\end{equation}
	Once the above is shown, then 
	we obtain
	\begin{equation*}
		\text{length}(\rd\Omg_{t})
		\geq|a-b|\geq a_{2}-b_{2}
		\geq C_{2}t,
	\end{equation*}
	and the proof is done.
	
	We consider  
	the point $ 
	\Phi(t,(0,0)
)\in \overline{\bbR_{+}^{2}}=\mathbb{R}_{\geq0}\times\mathbb{R}$. 
	Then note that we have $ \Phi_{1}(t,(0,0))=0 $ because any point on the $ x_{2}- $axis stays on the axis along the flow map, due to the tangential boundary condition of $ u $ at $ x_{1}=0 $.
	Additionally, using the inequality \eqref{eq_u1u2bd} from Corollary \ref{cor_uift}, we get
	\begin{equation*}
		\begin{split}
			\Phi_{2}(t,(0,0)
			)&=\Phi_{2}(0,(0,0)
			)+\int_{0}^{t}\frac{d}{ds}\Phi_{2}(s,(0,0)
			)ds=0
			+\int_{0}^{t}u_{ext,2}\big(s,\Phi(s,(0,0)
			)\big)ds\\
			&=
			\int_{0}^{t}u_{2}\Big(s,Q\big(\Phi(s,(0,0))\big)\Big)ds\geq\int_{0}^{t}\Big[\br{u}_{2}\Big(Q\big(\Phi(s,(0,0))\big)\Big)-C_{4}\dlt^{\frac{1}{4}}\Big]ds.
		\end{split}
	\end{equation*}
	Then due to $ \Phi_{1}(s,(0,0)
	)=0 $, 
	we have the lower bound of $ \Phi_{2}(t,(0,0))
	 $:
	\begin{equation}\label{eq_brx2lowbd}
		\begin{split}
			\Phi_{2}(t,(0,0))
			&\geq\int_{0}^{t}\Big[\br{u}_{2}\Big(Q\big(\Phi(s,(0,0))\big)\Big)-C_{4}\dlt^{\frac{1}{4}}\Big]ds
			=\int_{0}^{t}\Big[\Big(1-Q_{1}\big(\Phi(s,(0,0))\big)\Big)-C_{4}\dlt^{\frac{1}{4}}\Big]ds\\
			&=\int_{0}^{t}\big[\big(1-\Phi_{1}(s,(0,0))\big)-C_{4}\dlt^{\frac{1}{4}}\big]ds=\int_{0}^{t}(1-C_{4}\dlt^{\frac{1}{4}})ds=(1-C_{4}\dlt^{\frac{1}{4}})t.
		\end{split}
	\end{equation}
	Similarly, we can find the upper bound of $ k(1_{\Phi_{t}(\Omg_{0})}) $ using the inequality \eqref{eq_comgrowth} from Lemma \ref{lem_comgrowth} and the calculation \eqref{eq_derivofRt'1}:
	\begin{equation}\label{eq_k1Omguppbd}
		\begin{split}
			k(1_{\Phi_{t}(\Omg_{0})})&=k(1_{\Omg_{0}})+\int_{0}^{t}\frac{d}{ds}k(1_{\Phi_{s}(\Omg_{0})})ds\leq
			\int_{0}^{t}\bigg(\frac{d}{ds}k(1_{\br{\Phi}_{s}(\br{\Omg})})+C_{5}\dlt^{\frac{1}{4}}\bigg)ds=
			\bigg(\frac{1}{2}+C_{5}\dlt^{\frac{1}{4}}\bigg)t.
		\end{split}
	\end{equation}
	Then by the above bounds \eqref{eq_brx2lowbd}, \eqref{eq_k1Omguppbd}, and the condition $ (d) $ of \eqref{eq_condiofOmg}, we have
	\begin{equation*}\label{eq_x2kdiff}
		\begin{split}
			\Phi_{2}(t,(0,0)){-\frac{t}{4}}
			&\geq 
			(1-C_{4}\dlt^{\frac{1}{4}})t-\frac{t}{4}=
			\bigg(\frac{1}{2}-C_{4}\dlt^{\frac{1}{4}}\bigg)t+\frac{t}{4}\\
			&\geq
			\bigg(\frac{1}{2}-C_{4}\dlt^{\frac{1}{4}}\bigg)t+(C_{4}+C_{5})\dlt^{\frac{1}{4}}t\geq
			\bigg(\frac{1}{2}+C_{5}\dlt^{\frac{1}{4}}\bigg)t\geq k(1_{\Phi_{t}(\Omg_{0})}).
		\end{split}
	\end{equation*}
	We choose $ a=\Phi(t,(0,0))
	 $ and $ C_{2}=\frac{1}{4} $. It is only left to show that there exists a point $ b\in\rd\Phi_{t}(\Omg_{0}) $ such that
	\begin{equation}\label{eq_yt2}
		k(1_{\Phi_{t}(\Omg_{0})})\geq b_{2}.
	\end{equation}
	Indeed, if every point $ \xi\in\rd\Phi_{t}(\Omg_{0}) $ satisfies $ k(1_{\Phi_{t}(\Omg_{0})})<\xi_{2} $, then it leads to the following contradiction:
	\begin{equation*}
		k(1_{\Phi_{t}(\Omg_{0})})=
		\bigg(\int_{\Phi_{t}(\Omg_{0})}1\ dx\bigg)^{-1}\cdot\int_{\Phi_{t}(\Omg_{0})}x_{2}dx>\bigg(\int_{\Phi_{t}(\Omg_{0})}1\ dx\bigg)^{-1}
		\cdot k(1_{\Phi_{t}(\Omg_{0})})\cdot\int_{\Phi_{t}(\Omg_{0})}1\ dx
		=k(1_{\Phi_{t}(\Omg_{0})}).
	\end{equation*}
	This finishes the proof of the claim \eqref{eq_claimthm2}.
\end{proof}

$\bullet$ Data availability statement : Data sharing is not applicable to this article as no datasets were generated or analysed during the current study.

\subsection*{Acknowledgments}
 KC has been supported by the National Research Foundation of Korea (NRF-2018R1D1A1B07043065).
IJ has been supported  by the New Faculty Startup Fund from Seoul National University and the Samsung Science and Technology Foundation under Project Number SSTF-BA2002-04.

\subsection*{Conflict of interest}

The authors state that there is no conflict of interest.

\bibliography{Choi_Jeong_Lim_20220907_references}

\end{document}